\documentclass[11pt,a4paper]{article}
\usepackage[cp1251]{inputenc}
\usepackage{amsmath,amssymb,amsthm}

\usepackage{hyperref}
\usepackage{amscd}
\usepackage{euscript}

\newcommand{\E}{\mathrm E}

\newcommand{\V}{\mathrm V}

\newcommand{\R}{\mathbb R}

\newcommand{\C}{\mathcal C}

\renewcommand{\d}{\partial}

\newcommand{\Hh}{\mathrm {H}_0}
\newcommand{\te}{\mathrm {\theta}}
\newcommand{\Uu}{\EuScript U}

\newcommand{\s}{\mathfrak s}
\newcommand{\M}{\widetilde{M}}

\newcommand{\diam}{\operatorname{diam}}

\renewcommand{\int}{\operatorname{Int}}

\newcommand{\pr}{\operatorname{pr}}
\newcommand{\id}{\operatorname{id}}

\newcommand{\wt}{\widetilde}

\renewcommand{\phi}{\varphi}

\def\lasdim{\mathop{\rm {\ell\hbox{-}asdim}}\nolimits}
\def\asdim{\mathop{\rm {asdim}}\nolimits}

\theoremstyle{plain}
\newtheorem{theorem}{Theorem}%[subsection]
\newtheorem{corollary}{Corollary}%[subsection]
\newtheorem{lemma}{Lemma}%[subsection]
\newtheorem{proposition}{Proposition}%[subsection]
\theoremstyle{definition}
\newtheorem*{definition}{Definition}%[subsection]

\theoremstyle{remark}
\newtheorem{remark}{Remark}

 \newtheoremstyle{break}% name
   {9pt}%      Space above, empty = `usual value'
   {9pt}%      Space below
   {\itshape}% Body font
   {}%         Indent amount (empty = no indent, \parindent = para indent)
   {\bfseries}% Thm head font
   {.}%        Punctuation after thm head
   {\newline}% Space after thm head: \newline = linebreak
   {}%         Thm head spec

 \theoremstyle{break}

\begin{document}

\title{Quasi-isometric embedding of the fundamental group of an orthogonal graph-manifold into 
a product 
of metric trees}
\author{Alexander Smirnov\footnote{Supported by RFFI Grant
11-01-00302-a}}
\date{}
\maketitle

\begin{abstract}
In every dimension 
$n\ge 3$ 
we introduce a class of orthogonal graph-manifolds and prove 
that the fundamental group 
of any orthogonal graph-manifold quasi-isometrically 
embeds into 
a product of 
$n$ 
trees. 
As a consequence, we obtain 
that asymptotic and linearly-controlled asymptotic dimensions of such group are equal to 
$n$. 
\end{abstract}

\section{Introduction}
We introduce a class 
$\mathcal O$ 
of orthogonally glued higher-dimensional graph-manifolds 
(that we call throughout this paper {\em orthogonal graph-manifolds}; see section~\ref{main:dfn} for the definition).
Using the ideas of the paper~\cite{HuSi}, we generalize the results of that paper to the case of the class 
$\mathcal O$.

\begin{theorem}\label{main:thm}
For every $n$-dimensional orthogonal graph-manifold 
its fundamental group supplied with an arbitrary word metric  
admits a quasi-isometric embedding into 
a product of 
$n$ 
metric trees. 
As a consequence, asymptotic and linearly-controlled asymptotic dimensions of such group are equal to 
$n$.
\end{theorem}

In the paper~\cite{HuSi} this result was obtained in the 3-dimensional case for every graph-manifold in the sence 
of the definition in section~\ref{subsect:graph_manifolds}.
In fact, according to the paper~\cite{KL}, the fundamental group 
of any 3-dimensional graph-manifold is quasi-isometric to the fundamental group of some flip-manifold, 
which is precisely an orthogonal graph-manifold in the dimension 3. 
Also note that the inequality $\asdim \pi_1(M)\leq n$ for the fundamental group of 
an orthogonal graph-manifold $M$ follows from the result obtained in the Bell~--~Dranishnikov~\cite{BD}.

\section{Preliminaries}
\label{sect:definitions}

\subsection{Graph-manifolds}
\label{subsect:graph_manifolds}

\begin{definition}
A higher-dimensional graph-manifold is a closed, orientable, $n$-dimensional,
$n\geq 3$, 
manifold
$M$ 
that is glued from a finite number of blocks 
$M_v$, 
$M = \bigcup_{v\in V}M_v$. 
These should satisfy the following conditions (1)--(3).
\begin{itemize}
 \item[(1)] Each block $M_v$ is a trivial 
$T^{n-2}$-bundle
over a compact, orientable surface $\Phi_v$ 
with boundary (the surface should be different from the disk and the annulus),
where
$T^{n-2}$
is a
$(n-2)$-dimensional torus; 

 \item[(2)] the manifold  
$M$ is glued from blocks 
$M_v$, $v \in V$, 
by diffeomorphisms between boundary components  
(the case of gluing boundary components of the same block is not excluded);

 \item[(3)] gluing diffeomorphisms do not identify the homotopy
classes of the fiber tori.

\end{itemize}
\end{definition}

For brevity, we use the term ``graph-manifold''
instead of the term ``higher-dimensional graph-manifold''. 

Let
$G$
be a graph dual to the decomposition of
$M$ 
into blocks. The set of blocks of the graph-manifold coincides with the vertex set   
$\V=\V(G)$ 
of the graph  
$G$.
The set of (non-oriented) edges
$\E =\E(G)$
of
$G$
consists of pairs of glued components of blocks. We denote the set 
of the oriented edges of $G$ by $W$.

For more information about the graph-manifolds see~\cite{BK}.

\subsection{Orthogonal graph-manifolds}
\label{subsect:ort_graph_manifolds}

In this section we define a class of graph-manifolds that admit an orthogonally glued
metric of a special form. 
For brevity, we will call them orthogonal graph-manifolds.

Fix a graph $G$ and for each vertex $v \in \V(G)$ 
consider a surface $\Phi_v$
of nonnegative Euler characteristic with $|\d_v|$ boundary components, 
where $\d_v$ is the set of all edges
adjacent to the vertex
$v$.
Moreover, we assume that there is a bijection between the set of boundary components of 
the surface $\Phi_v$ and the set of all oriented edges adjacent to $v$.
For the block $M_v$ corresponding to a vertex $v$ we fix a trivialization 
$M_v=\Phi_v\times S^1\times\dots\times S^1$,
where
$\Phi_v$
is the  base surface, i.e. we fix simultaneously a trivialization
$M_v=\Phi_v\times T^{n-2}$
of
$M_v$
and a trivialization
$T^{n-2}=S^1\times\dots\times S^1$
of the fiber torus. For each block
$M_v$, 
we fix a coordinate system
$(x, x_1, \ldots, x_ {n-2})$
compatible with this decomposition, where
$x\in \Phi_v$ and $x_i\in [0,\,1)$ for each $1\leq i \leq n-2$.
For each oriented edge $w$ adjacent to the vertex $v$,
we define the coordinate system 
$(x_0)$, $x_0\in [0,\,1)$ 
on the corresponding component of the boundary  
$\partial \Phi_v$ 
of the surface  
$\Phi_v$.  
It defines the coordinate system $(x_0,\ldots,x_{n-2})$ 
on the boundary torus $T_ {w}$ of the block $M_v$.
Similarly, for the edge $-w$ inverse to the edge $w$ 
on the boundary torus $T_{-w}$,  
we define
the coordinate system
$(x'_0,\ldots,x'_{n-2})$. 
For each oriented edge $w$, we consider a permutation 
$\s_w$ of a 
well-ordered $(n-1)$-element set 
$(x_0,\dots,x_{n-2})$ such that $\s_w(x_0)\neq x_0$.
Furthermore, we assume
that for mutually inverse edges $w$ and $-w$ the permutations
$\s_ {w}$ and $\s_ {-w}$ are inverse ($\s_{-w}\circ \s_{w}=\id$). 
We define the gluing map $\eta_w\colon T_w\to T_{-w}$ by
$\eta_w((x_0,\ldots,x_{n-2}))=(\s_w(x_0),\ldots,\s_w(x_{n-2}))$. 
Note that this map is a well-defined gluing,  
as permutations $\s_w$ and $\s_{-w}$
are selected to be mutually inverse. 
Also, the map $\eta_w$ does not identify the homotopy
classes of fiber tori.

\begin{definition}\label{main:dfn}
The above described graph-manifold is called {\em an orthogonal
graph-manifold}.
\end{definition}

\begin{remark}
As mentioned above, in the case $n=3$ the class of all orthogonal graph-manifolds coincides 
with the class of all flip graph-manifolds considered in~\cite{KL}.
\end{remark}

\subsection{Metric trees}
\label{subsect:mtrees}

A {\em tripod} in a geodesic metric space $X$ is a union of three geodesic segments
$xt\cup yt\cup zt$ which have only one common point $t$. 
A geodesic metric space $X$ is called {\em a metric tree} if each
triangle in it is a tripod (possibly degenerate).

\subsection{Finitely generated groups}\label{subsect:fggroups}

Let $G$ be a finitely generated group and $S \subset G$ a finite symmetric
generating set for $G$ ($S^{-1} = S$). Recall that {\em a word metric} on the group 
$G$ (with respect to $S$) is the left-invariant metric defined by the norm 
$\|\cdot\|_S$, 
where for each
$g\in G$ its norm $\|g\|_S$ is the smallest number of elements of $S$ whose product
is $g$. It is known that all such metrics for the group $G$ are bi-Lipschitz 
equivalent (see~\cite{BBI}). In this paper we will consider only finitely generated 
groups with a word metric.

\subsection{Quasi-isometric maps}\label{subsect:qimaps}

A map $f\colon X\to Y$ is said to be {\em quasi-isometric} if there exist
$\lambda\geq 1,C\geq 0$ such that
$$\frac{1}{\lambda}|xy|-C\leq |f(x)f(y)|\leq \lambda |xy|+ C$$
for each $x,y\in X$. Metric spaces $X$ and $Y$ are called {\em quasi-isomeric} if there
is a quasi-isometric map $f\colon X\to Y$ such that $f(X)$ is a net in $Y$. 
In this case, $f$ is called {\em a quasi-isometry}.

\subsection{The metric on the universal cover}

Let us recall the famous Milnor--\v{S}varc Lemma.

\begin{lemma}\label{lem:qmetric}
Let $Y$ be a compact length space and let $X$ be the universal cover of $X$  
considered with the metric lifted from $Y$.
Then $X$ is quasi-isometric to the fundamental group $\pi_1(Y)$ 
of the space $Y$ considered with an arbitrary word metric.
\end{lemma}

It follows from this lemma that to prove Theorem~\ref{main:thm} it is sufficient 
to construct a quasi-isometric embedding of the universal cover into 
a product of $n$ trees.

\subsubsection{Metrics of non-positive curvature}
\label{subsubsect:mnpc}

Define a metric on the orthogonal graph-manifold
$M$ as follows:
for each edge
$e \in \E(G)$
take a flat metric
on its corresponding torus
$T_e$
such that any base circle of the coordinate system described above 
has length 1, and any two  of these circles are perpendicular. 

In particular, for each vertex
$v \in \V(G)$
there is a metric
on the boundary surface
$\Phi_v$
in which every boundary component has length 1.
This metric can be extended to a metric of nonpositive curvature
on the surface
$\Phi_v$
so that its boundary is geodesic.
Therefore, the metric from the boundary tori extends to the metric on the block
$M_v$, which is locally a product metric
(in general, a metric on the block may not be a product metric and it can 
have nontrivial holonomy along some loops on the base $\Phi_v$). 

Further we consider only those metrics on orthogonal graph-manifolds.
If we lift the above metric in the universal cover
$\M$, it follows from the Reshetnyak gluing theorem (see~\cite{BBI}) that 
the obtained metric space is nonpositively curved (or Hadamard) space. 

We fix an orthogonal graph-manifold $M$ with the metric described above.

\subsubsection{The standard hyperbolic surface with boundary
$\Hh$}
\label{subsubsect:standard_base}

Consider the hyperbolic plane $\mathbb{H}^2_{\kappa}$ 
having a curvature 
$-\kappa$ $(\kappa>0)$ 
such that the side of a rectangular equilateral
hexagon 
$\theta$ 
in the plane 
$\mathbb{H}^2_{\kappa}$
has length 1. Let $\rho$ be the distance between the
middle points of sides, which have a common adjacent side, $\delta$ the diameter
of $\te$. 
We mark each second side of $\theta$ (so we have marked three sides) and
consider a set $\Hh$ defined as follows. 
Take the subgroup $G_\theta$ of the isometry
group of 
$\mathbb{H}^2_{\kappa}$
generated by reflections in (three) marked sides of $\theta$ and let
$\Hh$ be the orbit of $\theta$ with respect to $G_\theta$. 
Then $\Hh$ is a convex subset in $\mathbb{H}^2_{\kappa}$
divided into hexagons that are isometric to $\theta$. 
Furthermore, the boundary
of $\Hh$ has infinitely many connected components each of which is a geodesic
$\mathbb{H}^2_{\kappa}$.
The graph $T_{bin}$ dual to the decomposition of $\Hh$ into hexagons is the
standard binary tree whose vertices all have degree three. Any metric space
isometric to $\Hh$ will be called a $\theta$-tree. 
Given a vertex $p$ of $T_{bin}$, we denote by
$\theta_p$ the respective hexagon in $\Hh$.

\begin{remark} 
In what follows,  we will consider $T_{bin}$ as the metric space with
a metric such that the length of each edge is equal to $2\rho$. Then these metric
spaces are metric trees. We will denote the set of vertices in $T_{bin}$ by $\V(T_{bin})$.
\end{remark}

\subsubsection{Standard metrics and bi-Lipschitz homeomorphisms between
bases}\label{subsubsect:standardblocks}

Consider a simplicial tree with the degree 3 of each vertex, 
and the length 1 of each edge.
We replace each edge by the rectangle $1\times 1/3^{100}$ and 
each vertex by the equilateral Euclidean triangle with each side equal 
to $1/3^{100}$. 
Then we glue them in a natural way.

\begin{definition}
The obtained metric space is called {\em a fattened tree} or {\em a standard surface}
and denoted by $X_0$.
\end{definition}

\begin{remark}
Note that the standard surface $X_0$ is bi-Lipschitz homeomorphic
to the $\theta$-tree $\Hh$.
So we fix an arbitrary bi-Lipschitz homeomorphism $h_0\colon \Hh\to X_0$.
\end{remark}

\begin{definition}
{\em The standard block} is defined to be a metric product of the  
$\theta$-tree 
$\Hh$ 
and $n-2$ copies of the Euclidean line $\R$, 
$B=\Hh\times\R\times\ldots\times\R$. 
\end{definition}

The partition of the $\theta$-tree by hexagons induces a partition of each boundary component
of the $\theta$-tree by unit segments.
For each boundary component such a partition is called {\em a grid} on this component. 
For each Euclidean factor $\R$ of the standard block, an arbitrary partition by unit segments 
is called {\em a grid} on this factor.
Finally for each boundary hyperplane $\sigma$ of the standard block 
a partition by unit cubes induced by grids on each factor is called {\em a grid} on this hyperplane.

Recall the following theorem.

\begin{theorem}[\cite{BN}, Theorem~1.2]\label{bnlemma}
Let $X_0$ be as above with a chosen boundary component $\d_0 X_0$. 
Then there exists $K>0$ and a function $\psi\colon \R \to \R$ 
such that for any $K_0>0$ and any $K_0$-bi-Lipschitz
homeomorphism $P_0$ from $\partial_0 X_0$ to a boundary component $\partial_1 X_0$, 
$P_0$ extends to a $\psi(K_0)$-bi-Lipschitz homeomorphism $P\colon X_0 \to X_0$ which
is $K$-bi-Lipschitz on every other boundary component.
\end{theorem}

\begin{corollary}\label{slebn}
Let 
$\wt{\Phi}_v$ 
be the universal cover of the surface 
$\Phi_v$
supplied with the metric described in sect.~\ref{subsubsect:mnpc}
with a chosen boundary component $\partial_0 \wt{\Phi}_v$.
Then there exists $K>0$ and a function $\psi\colon \R \to \R$ 
such that for any $K_0>0$ and any $K_0$--bi-Lipschitz
homeomorphism $P_0$ from $\partial_0\wt{\Phi}_v$ to a boundary component $\partial_0 \Hh$, 
$P_0$ extends to a $\psi(K_0)$-bi-Lipschitz homeomorphism $P\colon \wt{\Phi}_v \to \Hh$ which
is $K$--bi-Lipschitz on every other boundary component.
\end{corollary}
\begin{proof}
Rename the 
$K_0$, $K$ 
and 
$\psi$ 
from Theorem~\ref{bnlemma} to the 
$\bar{K}_0$, $\bar{K}$ 
and 
$\bar{\psi}$ 
respectively. 
Note that there is a bi-Lipschitz homeomorphism 
$\psi_v\colon X_0 \to \wt{\Phi}_v$. 

$$
\begin{CD}
X_0 @>>> X_0 \\
@V{\psi_v}VV @AA{h_0}A \\
\wt{\Phi}_v @>>> \Hh
\end{CD}
$$

Let $\psi_v$ be $M_1$-bi-Lipschitz and $h_0$ be $M_2$-bi-Lipschitz. 
We set $M:=\max\{M_1,M_2\}$.
Consider $P_0:\partial_0\wt{\Phi}_v\to\partial_0 \Hh$. 
Denote the boundary component $h_0(\d_0 H_0)$ by $\d_1 X_0$. 
Then  
if 
$\partial_0 X_0=\psi^{-1}_v(\d_0 \wt{\Phi}_v)$ the map 
$$h_0\circ P_0\circ \psi_v:\d_0 X_0\to\d_1 X_0$$ 
is $M^2\cdot K_0$-bi-Lipschitz homeomorphism. 
By Theorem~\ref{bnlemma} it extends to a
$\bar{\psi}(M^2\cdot K_0)$-bi-Lipschitz
homeomorphism that is $\bar{K}$-bi-Lipschitz
on each remaining boundary component of the space $X_0$.
Therefore, the homeomorphism 
$P=h^{-1}_0\circ \bar{P}\circ \psi^{-1}_v$ 
is  
$\psi(K_0)=M^2\bar{\psi}(M^2\cdot K_0)$-bi-Lipschitz. 
Moreover, on every other boundary component it is 
$K=M^2\cdot \bar{K}$-bi-Lipschitz.
Therefore, for $\psi(x)=\bar{\psi}(M^2\cdot x)$, $x\in \R$ and  
$K=M^2\cdot \bar{K}$ Corollary~\ref{slebn} is proved.
\end{proof}

\begin{remark}
Since the graph $G$ is finite, we can assume that the number $K$ and the function $\psi$
are independent of $v\in\V(G)$.
\end{remark}

\subsubsection{The special metric on the universal cover}\label{subsubsect:ucmetr}

In this section
we inductively construct for each orthogonal graph-manifold $M$ 
a special metric on its universal covering $\M$ 
so that $\M$ with such a metric is quasi-isometric to the fundamental group
$\pi_1(M)$ of the graph-manifold $M$.
Afterwards it will be sufficient to construct a quasi-isometric embedding 
of $\M$ into a product of $n$ trees. 

The decomposition of
$M$
into blocks lifts to a decomposition of
$\M$ 
into universal cover blocks, see \cite{BK}. 
We denote the tree dual to this decomposition by $T_0$.
Note that the degree of every vertex of
$T_0$
is infinite. On 
$T_0$,
we consider an intrinsic metric with length 1 edges.
Choose a vertex $o\in \V(T_0)$ in the tree  $T_0$ and call it {\em the root}. 
For each vertex $v \in \V(T_0)$, we define its rank $r(v)$ as 
the distance to $o$.
In particular, $r(o) = 0$.
For each vertex $v\in \V(T_0)$, we denote by $\M_v$
the corresponding block of the space $\wt{M}$.
This block is isometric to the product  
$\wt{\Phi}_v\times \R^{n-2}$.
Recall that on each boundary hyperplane of the block $\M_v$ 
we have fixed a coordinate system.
We call the axes of this system {\em the selected axes}.

By induction on the rank of vertices of the tree $T_0$,  
we construct on the space $\M$ a metric of special type, which is bi-Lipschitz homeomorphic 
to the metric lifted from the graph-manifold $M$.

\texttt{Base:}
Let
$\psi_o:\wt{\Phi}_o\to \Hh$
be a bi-Lipschitz homeomorphism.
Consider an isometric copy of the standard block 
$B_o$ 
and consider the map
$\psi'_o\colon \M_o\to B_o$
which is a direct product of the map
$\psi_o\colon \wt{\Phi}_o\to \Hh$ 
and the identity map 
$\id\colon \R^{n-2}\to \R^{n-2}$.

\texttt{Inductive step:} 
Suppose that for all vertices 
$v\in \V(T_0)$
such that
$r(v)\leq m$
we built a bi-Lipschitz homeomorphism 
$\psi'_v\colon \M_v\to B_v$ 
where $B_v$ is an isometric copy of the standard block.
Consider a vertex
$u\in \V(T_0)$ 
such that
$r(u)=m+1$. 
There is a unique vertex 
$v \in \V(T_0)$
adjacent to it such that $r(v)=m$.
Consider the blocks
$\M_v$ and $\M_u$ 
of the universal cover 
$\M$. 
Denote the covering map by $\pr \colon \M \to M$.
Recall that the gluing of the blocks
$\pr(\M_v)$ and $\pr(\M_u)$ 
is obtained by the permutation 
$\s$ 
of the coordinate system on the torus 
$ T_ {e} $.
Consider an isometric copy of the standard block. 
Denote it by $B_u$
and glue it to the block $B_v$ 
by the permutation $\s^{-1}$ of the coordinates along the corresponding hyperplanes
thus matching the grid on them.
By the induction, the map 
$\psi'_v\colon \M_v\to B_v$  
is the direct product of the map 
$\wt{\Phi}_v\to \Hh$ 
and $n-1$ maps $\R\to \R$.
Moreover, the restriction of each of these maps
to the intersection with the common boundary hyperplane 
of the blocks 
$\M_v$ and $\M_u$ 
is
a $K$-bi-Lipschitz homeomorphism onto its image. 
It follows from the orthogonality of the gluing that these restrictions induce 
a $K$-bi-Lipschitz homeomorphism 
$\partial_1 \psi_u\colon \partial_0 \wt{\Phi}_u\to \partial_0 \Hh$ 
from the boundary component $\d_0\wt{\psi}_u$ 
of the surface $\wt{\Phi}_u$
adjacent to the block
$\M_v$
to the boundary component 
$\d_0 \Hh$ of the $\theta$-tree
adjacent to the block 
$B_v$. 
Also, these restrictions induce 
a collection of $K$-bi-Lipschitz homeomorphisms
$\partial_i:\R\to \R$, 
each of which
maps the corresponding $\R$-factor 
of the decomposition
$\M_u=\wt{\Phi}_u\times\R\times\ldots\times\R$ 
to the corresponding 
$\R$-factor 
of the decomposition
$B_u=\Hh\times\R\times\ldots\times\R$.

By Corollary~\ref{slebn},  
the homeomorphism 
$\partial_1\psi_u$
extends to a 
$\psi(K)$-bi-Lipschitz homeomorphism
$\psi_u\colon \wt{\Phi}_u\to \Hh$, 
which
is $K$--bi-Lipschitz on every other boundary component.
We define the homeomorphism 
$\psi'_u$
as the direct product of the 
homeomorphism  
$\psi_u$ 
and $n-3$ homeomorphisms 
$\partial_i$ ($i=2,\ldots,n-2$). 

Let us construct a map
$\psi_M\colon\M\to X$, 
where $X$ is a metric space obtained by gluing blocks described above.
Namely, if the point 
$x$ 
lies in the block
$\M_v$
for some vertex
$v \in \V(T_0)$
we define $\psi_M(x):=\psi'_v(x)$.
The map 
$\psi_M$
is well defined, since the maps $\psi'_v$ are compatible with each other.

\begin{proposition}
The map constructed above is a bi-Lipschitz homeomorphism.
\end{proposition}
\begin{proof}
Let $C=\max\{K,\psi(K)\}$.
It follows from the construction that for each vertex
$u\in\V(T_0)$
the map
$\psi'_u$ 
is $C$-bi-Lipschitz.
Suppose that $x\in \M_v$ for some vertex $v \in \V(T)$
and $y\in \M_u$ for some vertex
$u\in\V(T)$.
Denote $x':=\psi_M(x)$ and $y':=\psi_M(y)$.

Let $\gamma$ be a geodesic between vertices $v$ and $u$ in the tree $T$. 
Denote its consecutive edges by 
$e_1,\ldots,e_k$.
Note that a geodesic 
$xy\subset\M$ 
consecutively intersects hyperplanes
$\sigma_1,\ldots,\sigma_k$
in the space 
$\M$ 
that correspond to these edges. 
Similarly, a geodesic
$x'y'\subset X$
consistently intersects hyperplanes 
$\sigma'_1,\ldots,\sigma'_k$
in the space $X$. 
Moreover, $\sigma'_i=\psi_M(\sigma_i)$.
Let $z_i$ be an intersection point of the geodesic $xy$ and the hyperplane $\sigma_i$.  
(We assume that $z_0 = x$, $z_{k +1} = y$.)

Let $z'_i=\psi_M(z_i)$.
Since for each vertex 
$v$
the restriction of the map 
$\psi_M$
on the block 
$\M_v$
is $C$-bi-Lipschitz and 
the points 
$z_i$ 
and 
$z_ {i +1}$ ($i =0, \ldots, k$) lie in the same block,
we have
$|z'_iz'_{i+1}|\leq C|z_iz_{i+1}|$.
Combining all these inequalities, we find that
$\sum\limits^{k}_{i=0}|z'_iz'_{i+1}|\leq C|xy|$. 
On the other hand, by the triangle inequality we have 
$|x'y'|\leq \sum\limits^{k}_{i=0}|z'_iz'_{i+1}|$. 
This implies that 
$|x'y'|\leq C|xy|$.
Similarly, we have $|xy|\leq C|x'y'|$. 
\end{proof}

Thus, we define a metric of special type on the universal cover 
of the orthogonal graph-manifold.
Such a metric has nonpositive curvature in the sense of Alexandrov.
Moreover, for every vertex 
$v\in\V(T_0)$  
the corresponding block
$\M_v$ 
is isometric to the direct product
of the $\te$-tree $\Hh$ 
and $n-2$ factors $\R$.

Let us introduce some technical notations that will be needed later.
For each vertex 
$v$ 
and the corresponding block of 
$\M_v$, 
denote by 
$X_v$
a copy of the corresponding $\te$-tree.
Moreover, denote a copy of the tree $T_ {bin}$ naturally (isometrically) embedded  
in the surface $X_v$,
considered with the above described metric, by $T_v$.
Let 
$G'_{\te}$ be the isometry group of the 
$\te$-tree.
Note that for each vertex 
$v$ there exists 
$2\delta$-Lipschitz retraction 
$r_v\colon X_v\to T_v$
equivariant under the action of 
$G'_\te$. 

Recall that the block $\M_v$ is 
a product $X_v\times\R\times\ldots\times\R$.  
Denote the projections to the corresponding factors
by $p^1_v,\ldots,p^{n-1}_v$.

For each point $x\in \M_v$
consider the map given by 
$$\pi_x (y):=(y, p^2_v(x), \ldots, p^{n-1}_v (x))\ \hbox{for every point $y \in X_v$.}$$

We call an orthogonal graph-manifold {\em irreducible}
if its universal
cover 
$\M$ 
is not a product of a Euclidean space and
universal cover of the orthogonal graph-manifold of lower dimension. 
It suffices to prove Theorem~\ref{main:thm} for the irreducible case.

\section{Trees $T_c$ and maps to them}
\label{sect:trees}

\subsection{Construction of trees}
\label{subsect:trees}

Let 
$\gamma = w_1\ldots w_k$ 
be an oriented path in the tree 
$T_0$.
Denote by 
$\s_{\gamma}$ the permutation 
$\s_{w_k} \circ \ldots \circ \s_{w_1}$
of well-ordered $(n-1)$-element set.

Suppose that vertices 
$u, v \in T_0$ 
are connected by two oriented paths
$\gamma_1$ and $\gamma_2$.
Note that 
$\s_{\gamma_1}=\s_{\gamma_2}$, 
therefore, we can define the permutation
$\s_{uv}$
as the permutation 
$\s_{\gamma}$ 
along any path 
$\gamma$ between 
$u$ 
and 
$v$. 
Furthermore, $\s_{v_1v_3}=\s_{v_2v_3}\circ  \s_{v_1v_2}$ and $\s_{uv}=\s_{vu}^{-1}$.

We define a relation $\sim$ on the set of vertices of $T_0$ by
$u\sim v$  
if and only if the permutation 
$\s_{uv}$
fixes the smallest element.
It is easy to check that the relation $\sim$ is an equivalence relation.

Let us prove that the relation 
$\sim$ 
divides the set 
$\V(T_0)$ 
into not more than $n-1$
equivalence classes.
Indeed, if it fails, then we can choose $n$ pairwise non-equivalent vertices
$v_1,\ldots,v_n$. 
Then, for some different 
$2\leq i,j\leq n$,  
we have
$\s_{v_1v_i}(x_0)=\s_{v_1v_j}(x_0)$, 
where
$x_0$ is the smallest element.
Then, since 
$\s_{v_iv_j}(x_0)=\s_{v_1v_j}\circ \s_{v_iv_1}(x_0)=x_0$,
$v_i\sim v_j$. 
This is a contradiction. 

Fix a vertex 
$u$ 
in the tree 
$T_0$. 
Since the manifold 
$M$ 
is irreducible, 
the set of permutations
$\{\s_{vu}\mid u,v\in \V(T_0)\}$ 
is transitive. 
That is, for each element 
$x$ 
of a well-ordered $(n-1)$-element set, there is a vertex  
$v\in \V(T_0)$
that 
$\s_{vu}(x_0)=x$.
It follows that there are at least 
$n-1$ 
different equivalence classes. 
Hence there are exactly 
$n-1$. 

Denote the set of all these classes by $\C$.
We have shown that 
$|\C| = n-1$. 
Given 
$c\in \C$, 
note that if 
$u,v\in c$ 
then for any vertex 
$v'\in \V(T_0)$  
we have 
$\s_{uv'}(x_0)=\s_{vv'}(x_0)$. 

Fix a vertex 
$v\in\V(T_0)$ 
and an equivalence class 
$c\in \C$. 
We construct a tree 
$T_{v,c}$ 
as follows.
If the vertex 
$v$ 
belongs to the class 
$c$
then we set 
$T_{v, c}:=T_v$, 
see the end of section~\ref{sect:definitions}. 
Otherwise, we set $T_{v,c}:=\R$.

Also, for each vertex 
$v\in \V(T_0)$ 
we construct a map
$r_{v,c}\colon \M_v\to T_{v,c}$.
If the vertex 
$v$ 
belongs to 
$c$
then we set
$r_{v,c}:=r_v\circ p^1_v$.
Otherwise, if the vertex 
$v$ 
belongs to some class
$c'\neq c$,
then we set 
$r_{v,c}=p^{k}_v$, 
where 
$k=\s_{uv}(x_0)$, 
$u\in c$, 
and 
$x_0$ is the smallest element. 

For each class 
$c\in \C$, 
we construct a tree of 
$T_c$ 
as follows.
For each pair of adjacent vertices 
$u, v\in\V(T_0)$,  
we say that a point 
$x\in T_{u, c}$ 
and a point 
$y \in T_{v, c}$
are
{\em $\sim_c$-equivalent} 
if there exists a point
$z\in \M_u\cap \M_v$ 
such that 
$x=r_{u,c}(z)=r_{v,c}(z)=y$. 

This relation is well defined. Indeed, for every point 
$x \in T_{u, c}$ 
the preimage
$r^{-1}_{u,c}(x)\cap\M_u\cap\M_v$ 
is an 
$(n-2)$-dimensional 
subspace orthogonal to the coordinate 
$\s_{v'u}(x_0)$, 
where the vertex 
$v'$ 
belongs to 
$c$.
Similarly, for each point
$y\in T_{v,c}$ 
the preimage 
$r^{-1}_{v,c}(x)\cap\M_u\cap\M_v$ 
is 
an 
$(n-2)$-dimensional 
subspace orthogonal to the coordinate
$\s_{v'v}(x_0)$.
But by the definition of coordinates 
$\s_{v'u}(x_0)$ 
and 
$\s_{v'v}(x_0)$, 
any two such subspaces are either disjoint or coincide.
This implies immediately the following lemma.

\begin{lemma}\label{lem:equiv}
Let 
$u, v\in \V(T_0)$ 
be a pair of adjacent vertices and 
$c\in \C$ 
be an equivalence class.  
If the points 
$x,x'\in T_{u,c}$ 
and 
$y\in T_{v,c}$ 
are such that 
$x\sim_c y$ 
and 
$x'\sim_c y$,
then 
$x=x'$.  
\end{lemma}

Extend the relation 
$\sim_c$ 
by transitivity.
This means that we set 
$x\sim_c y$  
if and only if 
there exists a chain
$x=x_0, \ldots, x_l=y$  
that 
$x_i\sim_c x_{i+1}$ 
and 
$x_i\in T_{v_i,c}$  
for each 
$0\leq i\leq l-1$, 
and the vertices 
$v_i$ 
and 
$v_{i+1}$ 
are adjacent in the tree $T_0$.  
From Lemma~\ref{lem:equiv}, it follows that the relation 
$\sim_c$ 
is an equivalence relation.

\begin{lemma}\label{lem:isom}
Let 
$c\in \C$ 
be an equivalence class.
Fix any pair of vertices
$u,v \in c$, 
and  consider points  
$x,y\in T_{u,c}$ 
and points 
$x',y'\in T_{v,c}$  
such that 
$x\sim_c x'$ 
and 
$y\sim_c y'$. 
Then 
$|xy|=|x'y'|$. 
\end{lemma}
\begin{proof}
In fact, let
$u=v_0,\ldots,v_l=v$ 
be consecutive vertices of the geodesic between vertices
$u$ 
and 
$v$ 
in the tree 
$T_0$.
Note that it suffices to consider the case when the vertices  
$v_1,\ldots,v_{l-1}$ 
are not in the class 
$c$. 
For each 
$0\leq i\leq l-1$ 
denote the common hyperplane of the blocks
$\M_{v_i}$ 
and 
$\M_{v_{i+1}}$ 
by 
$\sigma_i$. 
Let 
$\d X_u$ 
be the common boundary component of the 
$\te$-tree 
$X_u$ 
corresponding to the hyperplane 
$\sigma_0$. 
Well as, let  
$\d X_v$ 
be the common boundary component of the
$\te$-tree 
$X_v$ 
corresponding to the hyperplane 
$\sigma_{l-1}$.
By the construction of the metric on the universal cover
$\M$ 
for each interval 
$I_0$ 
of the grid on the boundary component 
$\d X_u$, 
there are segments
$I_1,\ldots,I_{l-2}$
of the grids on the on the corresponding 
$\R$-factors
and the segment 
$I_{l-1}$ 
of the grid on the boundary component
$\d X_v$ 
such that for each 
$1 \leq i \leq l-1$ 
we have 
$p^j_{v_i}(I_{i-1})=p^j_{v_i}(I_{i})$, 
where 
$j=\s_{uv_i}(x_0)$. 
Note that it is sufficient to prove the Lemma for arbitrary points
$x,y\in r_{u,c}(I_0)\subset T_{u,c}$.
But for such points we have
$x',y'\in r_{v,c}(I_{l-1})$, 
hence, by the equivariance of retractions
$r_u$ 
and 
$r_v$ 
the required equality 
$|xy| = |x'y'|$ 
is satisfied.
\end{proof}

Fix 
$c\in\C$.
Define the space 
$T_c$ 
as a factor
$\{\bigsqcup T_{v,c}\mid v\in V(T_0)\}/\sim_c$. 
From Lemma~\ref{lem:isom}, it follows that the resulting space is
a metric tree. 
Moreover, for each vertex
$v\in c$ 
the natural embedding 
$proj_{v,c}\colon T_{v,c}\to T_c$ 
is isometric. 

Thus, for each 
$c\in \C$
we constructed a tree 
$T_c$, 
which is naturally divided into blocks
$T_{v,c}$.

\subsection{Maps to trees}
\label{subsect:trmaps}

The remainder of this paper, we consider the product of 
$|\C|+1=n$ of constructed trees 
$T_0\times \prod\limits_{c\in \C} T_c$ 
as a metric space with the sum metric.
It means that the distance between two points
$x,y\in T_0\times \prod\limits_{c\in \C} T_c$
defined as the sum of the distances between their projections in the trees
$T_0$, $T_c$,  
for each 
$c\in \C$.
$$|xy|=|x_0y_0|_{T_0}+\sum\limits_{c\in\C}|x_cy_c|_{T_c}.$$

We define a map 
$\phi_c \colon \M \to T_c$ 
by the formula:
$$\phi_c(x):=(proj_{v,c}\circ r_{v,c})(x),$$ 
where 
$x\in \M_v$
and 
$proj_{v,c}$
is the natural embedding of the tree 
$T_ {v, c}$ 
in the tree 
$T_c$.
It follows from the definition of the maps 
$r_{v,c}$ 
and from the construction of the tree 
$T_c$ 
that the map 
$\phi_c$
is well defined on the intersection 
$\M_u\cap \M_v$ 
of each pair of adjacent blocks 
$\M_u$ 
and 
$\M_v$.  

From the definition of the map 
$r_{v,c}$,  
we have that 
$\phi_c$ 
is 
$2\delta$-Lipschitz.
To prove this, it suffices to show that
$$|\phi_c(x)\phi_c(y)|\leq 2\delta|xy|,$$
where 
$x$ 
and 
$y$ 
belong to the same block 
$\M_v$. 
This fact follows from the definition of 
$\phi_c$.

Define a map 
$\phi_0 \colon \M \to T_0$ 
as follows.
If 
$x\in \M_v$ 
and for any other vertex 
$u\in\V(T_0)$, 
$x\notin \M_u$  
set $\phi_0(x)=v$.
Otherwise if 
$x\in \M_u\cap \M_v$ 
and 
$r(u)<r(v)$  
set 
$\phi_0(x)=u$. 
It is clear that 
$|\phi_0(x)\phi_0(y)|\leq |xy|+1$.

Define a map 
$\phi\colon \M\to T_0\times \prod\limits_{c\in \C} T_c$  
by the equality 
$\phi:=\phi_0\times \prod\limits_{c\in \C}\phi_{c}$. 

Then  
$$|\phi(x)\phi(y)|=|\phi_0(x)\phi_0(y)|+\sum\limits_{c\in\C}|\phi_{c}(x)\phi_{c}(y)|\leq (2\delta(n-1)+1)|xy|+1.$$

\section{Special curves}
\label{sect:curves}

For a curve 
$\gamma$ 
in a metric space 
$X$ 
by 
$|\gamma|$ 
denote its length. 
For further proof we need the following lemma. 

\begin{lemma}\label{lem:curves}
Let $x,y\in \M$ be a pair of points.  
Then there exists a curve 
$\gamma\subset\M$
between them such that 
$|\gamma|\leq (2\delta+1)|\phi(x)\phi(y)|+2\delta$.
\end{lemma}

\begin{proof}
Consider vertices 
$u,v\in \V(T_0)$  
such that 
$x\in \M_v$, $y\in \M_u$. 
We prove the lemma by induction on the length of the path 
$|uv|_{T_0}$.

\texttt{Base:} 
There exists a vertex 
$v\in\V(T_0)$ 
that 
$x,y\in \M_v$.
In this case, the geodesic 
$xy$ 
does the job. Indeed,  
$$|xy|=\sqrt{|p^v_1(x)p^v_1(y)|^2+\ldots+|p^v_{n-1}(x)p^v_{n-1}(y)|^2},$$ 
that does not exceed 
$$2\delta+\sum\limits_{c\in\C}|\phi_{c}(x)\phi_{c}(y)|\leq (2\delta+1) |\phi(x)\phi(y)|+2\delta.$$

\texttt{Inductive step:}
Fix vertices 
$u, v\in\V(T_0)$. 
Let 
$\eta_0$ 
be a geodesic between 
$u$ 
and 
$v$ 
in the tree 
$T_0$.
Choose vertices 
$u', v'\in \eta_0$ 
that 
vertices 
$u$ 
and 
$u'$ 
as well as 
vertices 
$v$ 
and 
$v'$
are 
adjacent. 
Note that 
either 
$r(u)>r(u')$  
or 
$r(v)>r(v')$. 
Without loss of generality, assume that 
$r(v)>r(v')$.
We can also assume that the point 
$x$ 
does not belong to the block 
$\M_{v'}$.
In this case, 
$\phi_0(x)=v$ 
and 
and for any point 
$x'\in \M_{v}\cap\M_{v'}$ 
we have 
$\phi_0(x')=v'$. 
It follows that 
$|\phi_0(x)\phi_0(x')|=1$.

Assume the vertex 
$v$ 
belongs to the class 
$c\in \C$. 
Consider a geodesic 
$\eta$ 
between 
$\phi(x)$ 
and 
$\phi(y)$. 
Denote its projection to the tree 
$T_c$ 
by
$\eta_c$. 
Note that
the curve 
$\eta_c$ 
is a geodesic 
between 
$\phi_c(x)$ 
and 
$\phi_c(y)$.
Divide the curve 
$\eta_c$ 
into two parts
$\eta^1_c=\eta_c\cap T_{v,c}$ 
and 
$\eta^2_c=\eta_c\setminus \eta^1_c$.
Let 
$z$ 
be the end of the curve 
$\eta^1_c$  
different from  
$\phi_c(x)$.

Recall that 
$\pi_x \colon X_v \to \M_v$ 
is a horizontal embedding
such that the image contains the point 
$x$.
Further, we assume that the tree
$T_{v,c}$ 
is naturally embedded in the 
$\te$-tree 
$X_v$.  
Define a map 
$\pi'_x$ 
as the composition of such embedding and the map 
$\pi_x$. 
Denote the image of the curve 
$\eta^1_c$ 
under the map
$\pi'_x$ 
by 
$\gamma_c$. 
We set 
$x_0:=\pi'_x(\phi_c(x))$ 
and 
$x_1:=\pi'_x(z)$. 
Then 
$|xx_0| <\delta$ 
and there exists a point
$x_2\in \M_v\cap\M_{v'}$ 
such that 
$x_2\in \pi_x(X_v)$, $\phi_c(x_2)=z$ 
and for any 
$c\neq c'\in \C$ 
we have 
$\phi_{c'}(x_2)=\phi_{c'}(x)$.
Note that 
$|x_1x_2|<\delta$ 
and   
$$|\phi_{c'}(x)\phi_{c'}(x_2)|+|\phi_{c'}(x_2)\phi_{c'}(y)|=|\phi_{c'}(x)\phi_{c'}(y)|.$$ 
On the other hand, since 
$\phi_c(x_2)=z$,  
$$|\phi_{c}(x)\phi_{c}(x_2)|+|\phi_{c}(x_2)\phi_{c}(y)|
=|\phi_{c}(x)z|+|z\phi_{c}(y)|=
|\phi_c(x)\phi_c(y)|.$$
Finally, we note that the point 
$\phi_0(x_2)$ 
belongs to the geodesic between 
$\phi_0(x)$ 
and 
$\phi_0(y)$ 
in the tree 
$T_0$.  
It follows that 
$$|\phi_0(x)\phi_0(x_2)|+|\phi_0(x_2)\phi_0(y)|=|\phi_0(x)\phi_0(y)|.$$
It means that  
$$|\phi(x)\phi(x_2)|+|\phi(x_2)\phi(y)|=|\phi(x)\phi(y)|,$$ 
and 
$x_2\in \M_{v'}$.
By induction, for the points 
$x_2$ 
and 
$y$,  
there exists a curve 
$\gamma'$ between  
$x_2$ and 
$y$ 
with 
$$|\gamma'|\leq (2\delta+1) |\phi(x_2)\phi(y)|+2\delta.$$
Consider the curve 
$\gamma$ 
which is the union of the geodesic 
$xx_0$, 
the curve 
$\gamma_c$, 
the geodesic 
$x_1x_2$ 
and the curve 
$\gamma'$.   
We have  
$$|\gamma|=|xx_0|+|\gamma_c|+|x_1x_2|+|\gamma'|\leq
2\delta+|\gamma_c|+|\gamma'|\leq 2\delta+|\phi_c(x)\phi_c(x_2)|+|\gamma'|,$$ 
which by induction does not exceed 
$$2\delta+|\phi(x)\phi(x_2)|+(2\delta+1)|\phi(x_2)\phi(y)|+2\delta.$$ 
We have shown that 
$|\phi_0(x)\phi_0(x_2)|=1$, 
therefore, 
$$2\delta+|\phi(x)\phi(x_2)|=2\delta|\phi_0(x)\phi_0(x_2)|+|\phi(x)\phi(x_2)|\leq (2\delta+1)|\phi(x)\phi(x_2)|.$$
So  
$$|\gamma|\leq 2\delta+|\phi(x)\phi(x_2)|+(2\delta+1)|\phi(x_2)\phi(y)|+2\delta\leq$$
$$\leq (2\delta+1)(|\phi(x)\phi(x_2)|+|\phi(x_2)\phi(y)|)+2\delta,$$ 
hence  
$|\gamma|\leq (2\delta+1)(|\phi(x)\phi(y)|)+2\delta$.
\end{proof}
\begin{corollary}
For any points 
$x, y\in \M$ 
the inequality 
$|xy|\leq (2\delta+1)|\phi(x)\phi(y)|+2\delta$ 
holds.
\end{corollary}

Applying the above inequalities, we obtain 
$$|xy|/(2\delta+1)-2\delta/(2\delta+1)\leq |\phi(x)\phi(y)|\leq (2\delta(n-1)+1)|xy|+1,$$
therefore,  
$$1/C|xy|-1\leq |\phi(x)\phi(y)| \leq C|xy|+1,$$
where 
$C=\max\{2\delta+1, 2\delta(n-1)+1\}$.
This completes the proof of Theorem~\ref{main:thm}.

\section{Asymptotic dimensions}
\label{sect:asdims}

\subsection{Definitions}
\label{subsec:asdimdef}

Recall some basic definitions
and notations.
Let 
$X$ 
be a metric space.
We denote by 
$|xy|$
the distance between 
$x,y\in X$
and
$d(U,V):=\inf\{|uv|\mid u\in U,v\in V\}$
is the distance between 
$U,V\subset X$.

We say that a family 
$\EuScript U$ 
of subsets of
$X$ 
is a \textit{covering} if for each
point 
$x\in X$ 
there is a subset 
$U\in \EuScript U$
such that
$x\in U$.
A family 
$\EuScript U$ of sets
is \textit{disjoint} if each two sets
$U,V\in \Uu$ 
are disjoint. The union
$\EuScript U=\cup \{\EuScript U^{\alpha}\mid \alpha\in \EuScript A\}$
of disjoint families
$U^{\alpha}$
is said to be
$n$-{\em colored},
where
$n=|\EuScript A|$
is the cardinality of
$\EuScript A$.

Also, recall that a family 
$\EuScript U$ 
is
\textit{$D$-bounded}, if the diameter of every
$U\in \EuScript U$ 
does not exceed 
$D$, $\diam U\leq D$.
A $n$-colored family of sets 
$\EuScript U$ 
is
\textit{$r$-disjoint}, if for every color
$\alpha\in \EuScript A$ 
and each two sets
$U,V\in \EuScript U^{\alpha}$ 
we have 
$d(U,V)\geq r$.

The linearly-controlled asymptotic
dimension is a version of the Gromov's
asymptotic dimension, $\asdim$.

\begin{definition}(Gromov~\cite{Gro})
The \textit{asymptotic dimension}
of a metric space $X$, $\asdim X$, is the least
integer number $n$ such that for each sufficiently large
real $R$ there exists a $(n+1)$-colored,
$R$-disjoint, $D$-bounded covering of the space $X$,
where the number $D>0$ is independent of $R$.
\end{definition}

\begin{definition}(Roe \cite{Roe})
The \textit{linearly-controlled asymptotic dimension}
of a metric space $X$, $\lasdim X$, is the least
integer number $n$ such that for each sufficiently large
real $R$ there exists an $(n+1)$-colored,
$R$-disjoint, $CR$-bounded covering of the space $X$,
where the number $C>0$ is independent of $R$.
\end{definition}

It follows from the definition that   
$\asdim X\leq \lasdim X$ 
for any metric space 
$X$.

In the next section we show that the fundamental group of orthogonal graph-manifold satisfies  
$n\leq \asdim \pi_1(M)\leq \lasdim\pi_1(M)\leq n$.

\subsection{Upper and lower bounds}
\label{subsec:asdimestim}

Recall some properties of the above dimensions.

Let $X$ and $Y$ be metric spaces.
If 
$X$ 
is quasi-isometric to 
$Y$ 
then 
$\asdim X=\asdim Y$ 
and 
$\lasdim X=\lasdim Y$. 
If 
$X\subset Y$  
then 
$\lasdim X\leq \lasdim Y$. 
Also, 
$\lasdim X\times Y\leq \lasdim X+\lasdim Y$. 
Let 
$T$ 
be a metric tree, then 
$\lasdim T\leq 1$. 
It follows from the above properties, that 
$\asdim \pi_1(M)=\asdim \M \leq \lasdim \M\leq \lasdim (T_0\times \prod\limits_{c\in \C} T_{c})\leq n$.
On the other hand, the space 
$\M$ 
is an Hadamard manifold, and hence, see~\cite[Theorem~10.1.1]{BuySchr},
$\asdim \M\geq n$.


\begin{thebibliography}{9}                                                                                                

\bibitem{BN}  \textsc{Behrstock J. A. and  Neumann W. D.}:\
\textit{Quasi-isometric classification of graph manifold groups},
Duke Math. J., \textbf{Volume 141, Number 2} (2008), 217-240.

\bibitem {BBI}  \textsc{Burago D., Burago Y., Ivanov S.}:\
\textit{A Course in Metric Geometry}, AMS Bookstore, 2001.

\bibitem {BD}  \textsc{Bell G. and  Dranishnikov A.}:\ \textit{
On Asymptotic Dimension of Groups Acting on Trees}, Geometriae
Dedicata, \textbf{Volume 103, Number 1} (2004), 89-101.

\bibitem{BK} \textsc{Buyalo S. V., Kobel'skii V.L.}:
\textit{Generalized graphmanifolds of nonpositive curvature,}
St.~Petersburg Math. J. 11 (2000), 251--268.


\bibitem{BuySchr}  \textsc{Buyalo S. V., Schroeder V.}:\ \textit{Elements of asymptotic geometry}, EMS monographs in mathematics, European Mathematical Society, 2007.
 
\bibitem{Gro}  \textsc{Gromov M.}:\ \textit{Asymptotic invariants
of infinite groups}, London Mathematical Society Lecture Note Series,
\textbf{Volume 182} (1993), 1-295.

\bibitem{HuSi}  \textsc{Hume D., Sisto A,}:\ \textit{Embedding universal covers of 
graph manifolds in products of trees}, preprint arXiv:math.GT/1112.0263

\bibitem{KL}  \textsc{Kapovich M. and Leeb B.}:\ \textit{3-manifold groups and 
nonpositive curvature}, Geometric Analysis and Functional Analysis,
\textbf{Volume 8} (1998), 841-852.

\bibitem{Roe}  \textsc{Roe J.}:\ \textit{Lectures on Coarse Geometry},
University Lecture Series, \textbf{Volume 31}, AMS, 2003.



\end{thebibliography}
\end{document}